\documentclass[11pt,dvips,twoside,letterpaper]{article}
\RequirePackage[latin1]{inputenc} \RequirePackage[T1]{fontenc}
\usepackage{pslatex}
\usepackage{fancyhdr}
\usepackage{graphicx}
\usepackage{geometry}

\def\figurename{Figure} 
\makeatletter
\renewcommand{\fnum@figure}[1]{\figurename~\thefigure.}
\makeatother

\def\tablename{Table} 
\makeatletter
\renewcommand{\fnum@table}[1]{\tablename~\thetable.}
\makeatother

\usepackage{amsmath}
\usepackage{amssymb}
\usepackage{amsfonts}
\usepackage{amsthm,amscd}

\newtheorem{theorem}{Theorem}[section]
\newtheorem{lemma}[theorem]{Lemma}

\newtheorem{proposition}[theorem]{Proposition}
\theoremstyle{definition}

\theoremstyle{remark}
\newtheorem{remark}[theorem]{Remark}

\numberwithin{equation}{section}

\def\P{\mathbb P}
\def\R{\mathbb R}
\def\E{\mathbb E}
\def\D{\mathbb D}

\def\E{\mathbb E}
\def\N{\mathbb N}

\begin{document}

\title{\bfseries\scshape{A Numerical scheme for backward doubly stochastic differential equations}
\thanks{This work was supported by AUF post doctoral grant 07-08, Réf:PC-420/2460  and partially performed when the author visit Université Cadi Ayyad of Marrakech (Maroc).}}
\author{\bfseries\scshape Auguste Aman\thanks{augusteaman5@yahoo.fr; auguste.aman@univ-cocody.ci}\\
UFR Mathématiques et Informatique\\ Universit\'{e} de Cocody, \\BP 582 Abidjan 22, C\^{o}te d'Ivoire}

\date{}
\maketitle \thispagestyle{empty} \setcounter{page}{1}

\begin{abstract}
In this paper we propose a numerical scheme for the class of backward doubly stochastic (BDSDEs) with possible path-dependent terminal values. We prove that our scheme converge in the strong $L^2$-sense and derive its rate of convergence. As an intermediate step we derive an $L^2$-type regularity of the solution to such BDSDEs. Such a notion of regularity which can be though of as the modulus of continuity of the paths in an $L^2$-sense, is new.
\end{abstract}


\vspace{.08in} \noindent \textbf{Keywords}: Backward doubly SDEs; $L^{\infty}$-Lipschitz functionals; numerical scheme; $L^{2}$-regularity, regression estimation.

\section{Introduction}
\setcounter{theorem}{0} \setcounter{equation}{0}
In this paper we are interested in the following backward
doubly stochastic differential equations (BDSDEs, in short):
\begin{eqnarray}
Y_{t}&=&\xi+\int_{t}^{T}f(s,X_{s},Y_{s},Z_{s})\,
ds+\int_{t}^{T}g(s,X_{s},Y_{s})\overleftarrow{dB_{s}} -\int_{t}^{T}Z_{s}dW_{s},\label{FBSDE}
\end{eqnarray}
where $W$ and $B$ is two independent Brownian motion defined on $(\Omega,\mathcal{F},\P)$ a complete probability
space. This kind of equation has two different
directions of stochastic integrals: standard (forward) stochastic integral driven by $W$ and
backward stochastic one driven by $B$. Initiated by Pardoux-Peng \cite{PP1}, BDSDEs is connected to quasi-linear stochastic
partial differential equations (SPDEs, in short) in order to derive Feynman-Kac formula for SPDEs. In this setting, BDSDEs have been extensively studied  in the past decade. We refer the readers to the papers of Buckdahn and Ma \cite{BM1,BM2}, Aman et al. \cite{Amanal}, Aman \cite{Aman}, Bahlali et al.\cite{Bahlalial}, and reference therein for more information on both theory and applications, especially in mathematical finance and stochastic control, for such equations.
In contrast, there was little progress made in the direction of the numerical implementation of BDSDEs. In special case of BDSDEs ($g\equiv 0$) called BSDEs, many efforts have been made in this direction as well.

Up to now basically two types of schemes have been considered. Based on the theo-\newline retical
four-step scheme from \cite{MPY}, first type of numerical algorithms for BSDEs have been developed by Douglas et al. \cite{Dal} and more recently by Milstein and Tretyakov \cite{MT}. The main focus of these algorithms is the numerical solution of parabolic PDEs which is related to BSDEs.

A second type of algorithms works backwards through time and tries to tackle the stochastic
problem directly. Bally \cite{Bal} and Chevance \cite{Chevance} were the first to study this type of algorithm
with a (hardly implementable) random time partition under strong regularity assumptions. The works of Ma et al. \cite{Mal} and Briand et al. \cite{Bral} ares in the same spirit, replacing, however, the Brownian motion by a binary random walk. Recently, Zhang proved, in \cite{Zhang}, a new notion of $L^2$-regularity on the control part  $Z$ of the solution which allowed proof of convergence of this backward approach with deterministic partitions under rather weak regularity assumptions (see \cite{Zhang},\,\cite{Tal} and \cite{Gal}) for different algorithms. All numerical schemes provide alternative to construct algorithm for PDEs. To the best of our knowledge to date there has been no discussion in the literature concerning numerical algorithms in the spirit of the last three works cited above in the general case i.e $g\neq 0$. This constitutes an insufficiency when we know that almost all the deterministic problems in these applied fields (PDEs) have their stochastic counterparts (SPDEs).

In this paper, to correct this empty, our goal is to build a numerical scheme following the idea used by Bouchard and Touzi \cite {Tal} and study its convergence. These results were important from a pure mathematical point of view as well as from application in the world of finance. Particulary, this numerical scheme opens the way for possible algorithm for determining the price of options on financial assets whose dynamics is solution of SPDEs.

Similarly to the special case $g\equiv 0$, the main difficulty lies in the approximation of the "martingale integrand" $Z$. In fact, in a sense the problem
often comes down to the path regularity of $Z$. However, in case $g\neq 0$, this regularity becomes a natural question to ask. Therefore, the first main result in this paper is to derive the path regularity called $L^{2}$-regularity for BDSDEs with the terminal value $\xi$ is the form $\Phi(X)$, where $X$ and $\Phi(.)$ are respectively diffusion process and $L^{\infty}$-Lipschitz functional (see Section 3 for precise definition). The proof is heavily related to Girsanov's Transformation which exists in the BDSDEs case only if $g$ do not depend to $z$.

The above $L^{2}$-regularity result allow us to provide the rate of convergence of our numerical scheme which is different from the one constructed in \cite{Tal}. Indeed, since BDSDEs have two directions of integral, our  numerical scheme need at each step the conditional expectation with respect the filtration $\mathcal{F}_{t_i}^{\pi}$ defined by $\mathcal{F}^{\pi}_{t_i}=\sigma(X^{\pi}_{t_j},\; j\leq i)\vee\sigma(B_{t_j},\; j\leq i)$. However we obtain the same convergence rate.

The rest of this paper is organized as follows. In Section 2, we introduce some fundamental knowledge and assumptions of BDSDEs. Section 3 is devoted to $L^{2}$-regularity results. In Section 4, we built our numerical scheme and prove the rate of convergence. Finally in section 5, we focus some ideas for the regression approximation and give it convergence rate.

\section{Preliminaries}
\setcounter{theorem}{0} \setcounter{equation}{0}
Let $(\Omega,\mathcal{F},\mbox{I\hspace{-.15em}P})$ be a complete probability  spaces, and $T>0$ be fixed throughout this paper. Let $\{W_{t}, 0\leq t\leq T\}$ and $\{B_{t}, 0\leq t\leq T\}$ be two mutually independent standard Brownian motions processes, with values respectively in $\R^{d}$ and $\R^{\ell}$, defined on  $(\Omega,\mathcal{F},\mbox{I\hspace{-.15em}P})$. Let $\mathcal{N}$ denote the class of $\P$-null sets of $\mathcal{F}$. For each $t\in [0,T]$, we define
\begin{eqnarray*}
\widetilde{\mathcal{F}}_{t}=\mathcal{F}^{W}_{t}\vee\mathcal{F}_{t,T}^{B},
\end{eqnarray*}
where for any process $\displaystyle{\left(\eta_{s}: 0\leq
s\leq T\right)},\;\mathcal{F}
^{\eta}_{s,t}=\sigma\{\eta_{r}-\eta_{s}, s\leq r \leq t\}\vee\mathcal{N}, \,\,\mathcal{F}^{\eta}_{t}=\mathcal{F}^{\eta}_{0,t}$.

We note that since the collection $(\widetilde{\mathcal{F}}_{t})_{t\geq 0}$ is neither increasing nor decreasing,
it does not constitute a filtration. Therefore, we define the filtration $(\mathcal{F}_{t})_{t\geq 0}$ by
\begin{eqnarray*}
\mathcal{F}_{t}=\mathcal{F}^{W}_{t}\vee\mathcal{F}_{T}^{B},
\end{eqnarray*}
which contains $\tilde{\mathcal{F}}_{t}$ and play a key role in the build of our numerical scheme.

For any real $p\geq 2$ and $k\in \N^{*}$, let $\mathcal{S}^{p}(\R^{k})$ denote
the set of jointly measurable processes $
 \{X_{t}\}_{t\in \lbrack 0,T]}$ taking values in $\R^{k}$ which satisfy
\begin{description}
\item $(i)\;\displaystyle{
\|X\|_{\mathcal{S}^{p}}=\E\left( \sup\limits_{0\leq t\leq
T}|X_{t}|^{p}\right) ^{\frac{1}{p}}<+\infty}$;
\item $(ii)$\,$X_t$  is $\widetilde{\mathcal{F}}_t$-measurable, for any $t\in[0,T]$.
\end{description}
We denote similarly by $\mathcal{M}^{p}(\R^{k})$ the set of (classes of $d\P\otimes dt$ a.e.
equal) $k$-dimensional jointly measurable processes which satisfy
\begin{description}
\item $(i)\,\displaystyle{
\|X\|_{\mathcal{M}^{p}}=\E \left[ \left(
\int_{0}^{T}|X_{t}|^{2}dt\right) ^{\frac{p}{2}}\right] ^{\frac{1}{p}%
}<+\infty};$
\item $(ii)$\,$X_t$  is $\widetilde{\mathcal{F}}_t$-measurable,\, for a.e. $t\in[0, T]$.
\end{description}
We denote by
\begin{description}
\item $\bullet\, W^{1,\infty}(\R^{k})$ the space of all measurable functions $\psi:\R^k\rightarrow\R$, such that for some constant $K>0$ it holds that $\displaystyle{|\psi(x)-\psi(y)|\leq K|x-y|,\; \forall\, x,\,y \in \R^k}$;
\item $\bullet\, \D$ the space of all càdlàg functions defined on $[0,T ]$;
\item $\bullet\, C^m_b([0,T]\times\R^k)$ the space of all continuous functions (not necessary bounded) $\psi:[0,T]\times\R^k\rightarrow\R$, such
that $\psi$ has uniformly bounded derivatives with respect to the spatial variables up to order $m$. We often denote $C^m_b = C^m_b ([0,T]\times\R^k)$ for simplicity, when the context is clear.
\end{description}
Let
\begin{eqnarray*}
b&:&[0,T]\times\R^{d}\rightarrow\R^{d},\\
\sigma&:&[0,T]\times\R^{d}\rightarrow\R^{d\times d},\\
f&:&[0,T]\times\R^{d}\times\R\times\R^{d}\xrightarrow{}\R,\\
g&:&\, [0,T]\times\R^{d}\times\R\xrightarrow{}\R^{\ell }
\end{eqnarray*}
be the functions satisfying the following assumptions: there exists constant $K>0$ such that for all $s,s'\in [0,T],\, x,x'\in\R^d,\; y, y'\in\R,\; z,z'\in\R^d$,
\begin{description}
\item $({\bf H1})\, |b(s,x)-b(s,x')|+\|\sigma(s,x)-\sigma(s,x')\|\leq K|x-x'|$.
\item $({\bf H2})$\\ $
\begin{array}{l}
(i)\; |f(s,x,y,z)-f(s',x',y',z')|^{2}\leq K\left(|s-s'|^{2}+|x-x'|^{2}+|y-y'|^{2}+|z-z'|^{2}\right),\\\\
(ii)\; |g(s,x,y)-g(s,x',y')|^{2} \leq  K(|s-s'|^{2}+|x-x'|^2+|y-y'|^{2}).
\end{array}$
\item $({\bf H3})\;\; \sup_{0\leq t\leq T}\{|b(t,0)|+|\sigma(t,0)|+|f(t,0,0,0)|+|g(t,0,0)|\}\leq K.
$
\end{description}
Given $\xi\in L^2(\Omega,\tilde{\mathcal{F}}_T,\P;\R^d)$, denote $(X,Y,Z)$ be the solution to the following FBDSDE:
\begin{eqnarray}
X_{t}&=&x+\int_{0}^{t}b(s,X_{s})ds+\int_{0}^{t}\sigma(s,X_{s})dW_{s}\label{FSDE}\\
Y_{t}&=&\xi+\int_{t}^{T}f(s,X_{s},Y_{s},Z_{s})\,
ds+\int_{t}^{T}g(s,X_{s},Y_{s})\overleftarrow{dB_{s}} -\int_{t}^{T}Z_{s}dW_{s}.\label{FBSDE}
\end{eqnarray}

Let now  recall some standard results appear in SDEs and BDSDEs literature.
\begin{proposition}[Karatzas and Shreve \cite{KS}]
\label{1}
Assume $(\bf H1)$ holds. Then for any initial condition $x\in\R^{d}$, FSDE \eqref{FSDE}
has a unique solution $(X_{t})_{0\leq t\leq T}$ belong to $\mathcal{S}^{p}(\R^d)$.

Moreover, for any $p\geq 2$, there exists a constant $C_p>0$, depending only on $T, K$ and $p$, such that
\begin{eqnarray*}
\E\left(\sup_{0\leq t\leq T}|X_t|^p\right)\leq C_p\left(|x|^{p}+\int^{T}_{0}[|b(t,0)|^p +|\sigma(t,0)|^p]dt\right)
\end{eqnarray*}
and
\begin{eqnarray*}
\E\left[|X_t-X_s|^{p}\right]\leq C_p\E\left(|x|^{p}+\sup_{0\leq t\leq T}|b(t,0)|^p+\sup_{0\leq t\leq T}|\sigma(t,0)|^p\right)|t-s|^{p/2}.
\end{eqnarray*}
\end{proposition}
\begin{proposition}[Pardoux and Peng \cite{PP1}]
\label{L2}
Under assumption $({\bf H2})$, BDSDE $(\ref{FBSDE})$ has a unique solution $(Y_t,Z_{t})_{0\leq t\leq T}$ in $\mathcal{S}^{p}(\R)\times\mathcal{M}^{p}(\R^d)$.

Moreover, for any $p\geq 2$, there exists a constant $C_p>0$, depending only on $T, K$ and $p$, such that
\begin{eqnarray*}
\E\left(\sup_{0\leq t\leq T}|Y_t|^p+\left(\int^{T}_{0}|Z_s|^{2}ds\right)^{p/2}\right)\leq C_{p}\E\left(|\xi|^p+\int^{T}_{0}[|f(t,0,0,0)|^p +|g(t,0,0)|^p]dt\right)
\end{eqnarray*}
and
\begin{eqnarray*}
\E\left[|Y_t-Y_s|^{p}\right]\leq C_{p}\E\left\{\left[|\xi|^{p}+\sup_{0\leq t\leq T}|f(t,0,0,0)|^p+\sup_{0\leq t\leq T}|g(t,0,0)|^p\right]|t-s|^{p-1}+\left(\int^{t}_{s}|Z_s|^{2}ds\right)^{p/2}\right\}.
\end{eqnarray*}
\end{proposition}
\begin{remark}
\begin{description}
\item In Proposition \ref{L2}, existence and uniqueness result needs only Lipschitz condition on $f$ and $g$ with respect variables $y$ and $z$
uniformly in $t$ and $x$.
\end{description}
\end{remark}
\begin{proposition}[Stability]
Let $(X^{\varepsilon},Y^{\varepsilon},Z^{\varepsilon})$ be the solution to the perturbed FBDSDE \eqref{FSDE} and \eqref{FBSDE} in which the coefficients are replaced by $b^{\varepsilon},\, \sigma^{\varepsilon},\, f^{\varepsilon},\, g^{\varepsilon}$, with initial state $x^{\varepsilon}$ an terminal value $\xi^{\varepsilon}$. Assume that the assumption $({\bf H1})$ and $({\bf H2})$ hold for all coefficients $b^{\varepsilon},\, \sigma^{\varepsilon},\, f^{\varepsilon},\, g^{\varepsilon}$ and assume that $\lim_{\varepsilon\rightarrow 0}x^{\varepsilon}=x$, and for fixed $(x,y,z)$,
\begin{eqnarray*}
\lim_{\varepsilon\rightarrow 0}\E\left\{\int_0^T[|b^{\varepsilon}(t,x)-b(t,x)|^2+|\sigma^{\varepsilon}(t,x)-\sigma(t,x)|^2]dt\right\}&=&0,\\
\lim_{\varepsilon\rightarrow 0}\E\left\{|\xi^{\varepsilon}-\xi|^2+\int_0^T[|f^{\varepsilon}(t,x,y,z)-f(t,x,y,z)|^2
+|g^{\varepsilon}(t,x,y)-g(t,x,y)|^2]dt\right\}&=&0.
\end{eqnarray*}
Then, we have
\begin{eqnarray*}
\lim_{\varepsilon\rightarrow 0}\E\left\{\sup_{0\leq t\leq T}|X^{\varepsilon}_t-X_t|^2+\sup_{0\leq t\leq T}|Y^{\varepsilon}_t-Y_t|^2+\sup_{0\leq t\leq T}|Z^{\varepsilon}_t-Z_t|^2\right\}=0.
\end{eqnarray*}
\end{proposition}

\section{$L^{2}$-regularity result for BDSDEs}
In this section we establish the first main result of this
paper which is $L^2$-regularity of the martingale integrand $Z$; and can be thought of as the modulus of continuity of the paths in an
$L^2$ sense. Such a
regularity, combined with the estimate for $X$ and $Y$, plays a key role
for deriving the rate of convergence of our numerical scheme in Section 4. We shall consider a class of BDSDEs with terminal values which are path dependent i.e of the form $\xi= \Phi(X)$, where a deterministic functional $\Phi:\D\rightarrow \R$ satisfies:
\begin{description}
\item ({\bf H4}) ($L^{\infty}$-Lipschitz condition). there exists a constant $K$ such that
\begin{eqnarray}
|\Phi(X_1)-\Phi(X_2)|\leq K\sup_{0\leq t\leq T}|X_{1}(t)-X_{2}(t)|,\;\;\;\; \forall\; X_1,\,X_2\in\D.\label{Lipschitzinfni}
\end{eqnarray}
\item ({\bf H5})\;$\Phi({\bf 0})$ is bounded by $K$, where ${\bf 0}$ denotes the constant function taking value $0$ on $[0,T]$.
\end{description}
This approximation due to Ma and Zhang \cite{MZ}, for
$L^{\infty}$-Lipschitz functional will be useful in the sequel.
\begin{lemma}
\label{Lemap}
Suppose $( {\bf H4})$ and $({\bf H5})$ hold. Let $\Pi = \left\{\pi\right\}$ be a family of partitions of $[0,T ]$. Then there exists a family of discrete functionals $\left\{h^{\pi}:\;\;\pi\in\Pi\right\}$ such that
\begin{description}
\item $(i)$ for each $\pi\in\Pi$, assuming $\pi: 0=t_0 < t_1 < t_2 < \cdot\cdot\cdot < t_n= T$, we have that $h^{\pi}\in C^{1}_b(\R^{d(n+1)})$, and satisfies:
\begin{eqnarray}
\sum_{i=1}^{n}|\partial_{x_i} h^{\pi}(x)|\leq K,\;\;\; \forall,\; x=(x_0,x_1,\cdot\cdot\cdot,x_n)\in\R^{d(n+1)},\label{bound}
\end{eqnarray}
where $K$ is the same constant as that in \eqref{Lipschitzinfni}.
\item $(ii)$ for any $X\in\D$, it holds that
\begin{eqnarray}
\lim_{|\pi|\rightarrow 0}|h^{\pi}(X_{t_0},X_{t_1},\cdot\cdot\cdot,X_{t_n})-\Phi(X)|=0,\label{specialconv}
\end{eqnarray}
where $|\pi|=\max_{1\leq i\leq n}|t_{i}-t_{i-1}|$.
\end{description}
\end{lemma}
Our main result in this section is the following theorem.
\begin{theorem}
\label{L05}
Assume $({\bf H1})$-$({\bf H5})$. Let $\pi_{0}: s_0<\cdot\cdot\cdot,s_m$ be any partition of $[0,T]$, and for each $1\leq i\leq m$, let define
\begin{eqnarray}
\tilde{Z}^{\pi_0}_{s_{i-1}}&=&\frac{1}{s_i-s_{i-1}}\E\left[\int^{s_i}_{s_{i-1}}Z_{s}ds|
\mathcal{F}_{s_{i-1}}\right].
\label{apz}
\end{eqnarray}
Then there exists a constant $C$ depending  only on $T$ and $K$, such that
\begin{eqnarray}
\E\left[\max_{1\leq i\leq m}\sup_{s_{i-1}\leq t\leq s_i}|Y_{t}-Y_{s_{i-1}}|^{2}
+\sum_{i=1}^{m}\int^{s_i}_{s_{i-1}}|Z_{s}
-\tilde{Z}^{\pi}_{s_{i-1}}|^{2}ds\right]\leq C|\pi_0|.
\label{a"61}
\end{eqnarray}
\end{theorem}
In the sequel, let $\pi: 0=t_0,\cdot\cdot\cdot,t_n=T$ be any partition of $[0,T]$ finer than $\pi_0$ and
without loss of generality, we assume $s_i = t_{l_i}$ for $i = 1, . . .,m$. Since $\Phi$ satisfies the $L^{\infty}$-Lipschitz condition \eqref{Lipschitzinfni}, by virtue of Lemma \ref{Lemap} one can find $h^{\pi}\in C^{1}(\R^{d(n+1)})$ satisfying \eqref{bound} and \eqref{specialconv}. Let $(Y^{\pi},Z^{\pi})$ be the solution to BDSDE:
\begin{eqnarray}
\label{eqpart0}
Y_{t}^{\pi}&=&h^{\pi}(X_{t_0},...,X_{t_n})+\int_{t}^{T}f(s,X_s,Y^{\pi}_s,Z^{\pi}_s)\,
ds+\int_{t}^{T}g(s,X_s,Y^{\pi}_s)\overleftarrow{dB_{s}} -\int_{t}^{T}Z_{s}^{\pi}dW_{s}.
\end{eqnarray}
Moreover,  setting $\Theta^{\pi}=(\Xi^{\pi},Z^{\pi})$, with $\Xi^{\pi}=(X,Y^{\pi})$, let $(\nabla X,\nabla^{i}Y^{\pi},\nabla^{i}Z^{\pi})$ be the unique solution of the following FBDSDE:
\begin{eqnarray}
\nabla X_t&=&I_d+\int^{t}_{0}b_x(r,X_{r})\nabla X_{r}dr+\int^{t}_{0}\sigma_x(r,X_{r})\nabla X_{r}dW_r, \nonumber\\
\nabla^{i}Y_{t}^{\pi}&=&\sum_{j\geq i}^{n}\frac{\partial h^{\pi}}{\partial x_j}(X_{t_0},...,X_{t_n})\nabla X_{t_j}
+\int^{T}_{t}[f_{x}(\Theta^{\pi}_r)\nabla
X_{r} +f_{y}(\Theta^{\pi}_r)\nabla^{i}Y_{r}^{\pi}
+f_{z}(\Theta^{\pi}_r)\nabla^{i}Z_{r}^{\pi}]dr\label{eqvaria}\\
&&+\int^{T}_{t}[g_x(\Xi^{\pi}_r)\nabla X_{r}+g_y(\Xi^{\pi}_r)\nabla^{i} Y_{r}^{\pi})]\overleftarrow{dB}_{r}-\int^{T}_{t}\nabla^{i}
Z_{r}^{\pi}dW_{r}, \;\; t\in[t_{i},T],\,\; i = 0,...,n-1.\nonumber
\end{eqnarray}

We denote
\begin{eqnarray*}
\xi^{0}&=&\int_{0}^{T}f_{x}({\Theta}^{\pi}_r)\nabla X_{r}N^{-1}_{r}dr+\int_{0}^{T}g_x({\Xi}^{\pi}_r)
\nabla X_{r}N^{-1}_{r}\overleftarrow{dB}_{r};\\
\xi^{i}&=&h^{\pi}(X_{t_0},\cdot\cdot\cdot,X_{t_n})\nabla X_{t_i}N^{-1}_{T},\; i=1,\cdot\cdot\cdot,n,
\end{eqnarray*}
where
\begin{eqnarray}
N_{t}&=&\exp\left(\int_{0}^{t}f_y({\Theta}^{\pi}_r)dr+\int_{0}^{t}g_y({\Xi}^{\pi}_r) \overleftarrow{dB}_r
-\frac{1}{2}\int_0^t |g_y({\Xi}^{\pi}_r)|^{2}dr\right),\nonumber\\
M_t&=&\exp\left\{\int_{0}^{t}f_z({\Theta}^{\pi}_r)dW_r-\frac{1}{2}\int_0^t |f_z({\Theta}^{\pi}_r)|^{2}dr\right\}.\label{m}
\end{eqnarray}
The following technical lemma is the building block of the proof of Theorem \ref{L05}.
\begin{lemma}
\label{L3}
Let consider the partition $\pi$ defined above and $h^{\pi}$ given by Lemma \ref{Lemap}, and assume $\sigma,\,b,\, f,\,g,\,\in C^{1}_b$. Then for all $i=1,...,n$
\begin{eqnarray*}
\nabla^i {Y}_{t}^{\pi}=\left(\xi_{t}^{0}+\sum_{j\geq i}\xi^{j}_{t}\right)M^{-1}_t N_{t}
-\int_{0}^{t}f_{x}({\Theta}^{\pi}_r)\nabla X_{r}N^{-1}_{r}drN_{t}-\int_{0}^{t}g_x({\Xi}^{\pi}_r)\nabla X_{r}N^{-1}_{r}\overleftarrow{dB}_rN_{t},
\end{eqnarray*}
where $\xi_{t}^{j}=\E\left(M_T\xi^{j}|\mathcal{F}_{t}\right),\; j=0,\cdot\cdot\cdot,n$.
\end{lemma}
\begin{proof}
For each $0\leq i\leq n$, we recall $(\nabla X,\nabla^{i}{Y}^{\pi},\nabla^{i}{Z}^{\pi})$, the solution of the linear FBDSDE \eqref{eqvaria}. Let $(\gamma^{0},\zeta^{0})$ and $(\gamma^{j},\zeta^{j}),\; j=1,\cdot\cdot\cdot,n$ be the solution of the BDSDEs
\begin{eqnarray}
\gamma^{0}_{t}&=&\int^{T}_{t}[f_{x}({\Theta}^{\pi}_r)\nabla X_{r} +f_{y}({\Theta}^{\pi}_r)\gamma_{r}^{0}+f_{z}({\Theta}^{\pi}_r)\zeta^{0}_{r}]dr
\nonumber\\
&&+\int^{T}_{t}[g_x({\Xi}^{\pi}_r)\nabla X_{r}+g_y({\Xi}^{\pi}_r)\gamma^{0}_{r}]\overleftarrow{dB}_{r}-\int^{T}_{t}\zeta^{0}_{r}dW_{r};\label{c8}\\
\gamma_{t}^{j}&=& \frac{\partial h^{\pi}}{\partial x_j}(X_{t_0},.....,X_{t_n})\nabla X_{t_j}+\int^{T}_{t}[f_{y}({\Theta}^{\pi}_r)\gamma_{r}^{j}
+f_{z}({\Theta}^{\pi}_r)\zeta_{r}^{j}]dr\nonumber\\
&&+\int^{T}_{t} g_y({\Xi}^{\pi}_r)\gamma_{r}^{j}\overleftarrow{dB}_{r}-\int^{T}_{t}\zeta_{r}^{j}dW_{r},\nonumber
\end{eqnarray}
respectively, then we have the following decomposition:
\begin{eqnarray}
\nabla^{i}{Y}_{s}^{\pi}=\gamma_{s}^{0}+\sum_{j=i}^{n}\gamma_{s}^{j},\;\;\;\;\;\;s\in [t_{i-1},t_i).\label{c9}
\end{eqnarray}
Recall \eqref{m} and since $f_y,\, f_z$ and $g_y$ are uniformly bounded, by Girsanov's Theorem (see, e.g., \cite{KS}) we know that $M$ is a $\P$-martingale on $[0,T]$, and $\widetilde{W}_{t}=W_{t}-\int_{0}^{t}f_z ({\Theta}^{\pi}_r)dr,\,\;\; t\in[0,T]$, is an $\mathcal{F}_t$-Brownian motion on the new probability
space $(\Omega,\mathcal{F},\widetilde{\P})$, where $\widetilde{\P}$ is defined by $\frac{d\widetilde{\P}}{d\P}=M_T$.

Now for $0\leq i\leq n$, define
\begin{eqnarray*}
\widetilde{\gamma}^{i}_t=\gamma^{i}_tN^{-1}_t,\;\;\;\;\;\;\;\;\widetilde{\zeta}^{i}_t=\zeta^{i}_tN^{-1}_t,\;\;\;\;\;t\in[0,T].
\end{eqnarray*}
Then, using integration by parts and equation \eqref{c8} we have
\begin{eqnarray*}
\widetilde{\gamma}^{i}_{t}=\xi^{i}
-\int_{t}^{T}\widetilde{\zeta}^{j}_rd\widetilde{W}_{r}, \;\; t\in[0,T].
\end{eqnarray*}
Therefore, by the Bayes rule (see e.g, \cite{KS} Lemma 3.5.3) we have, for $t\in[0,T]$
\begin{eqnarray*}
\gamma^{i}_{t}=\widetilde{\gamma}^{i}_{t}N_t&=&\E\left(M_T\xi^{i}|\mathcal{F}_{t}\right)M^{-1}_tN_t
=\xi_{t}^{i} M^{-1}_tN_t,
\end{eqnarray*}
where for $ i=0,\cdot\cdot\cdot,n$,
\begin{eqnarray}
\xi_t^{i}=\E\{M_T\xi^i|\mathcal{F}_t\}=\E(M_T\xi^i)+\int^{t}_{0}\eta^{i}_sdW_s.\label{mart}
\end{eqnarray}
Recalling \eqref{SE}, $M_T\in L^{p}(\mathcal{F}_T)$ and $\nabla X\in L^{p}({\bf F})$ for all $p\geq 2$.\newline Therefore for each $p\geq 1$,\eqref{bound} leads to
\begin{eqnarray}
\E\left\{\sum_{j=1}^{n}|M_T\xi^j|\right\}^{p}\leq C_p\E\left\{|M_T|^p\sup_{0\leq t\leq T}|\nabla X_t|^p\right\}.\label{major}
\end{eqnarray}
In particular, for each $j,\; M_T\xi^j\in L^{2}(\mathcal{F}_{T})$, thus \eqref{mart} makes sense.
Finally the result follows by $(\ref{c9})$.
\end{proof}

\begin{proof} [Proof of Theorem \ref{L05}]
Let us recall the partition $\pi_0$ defined above and consider $|\pi_0|$ it mesh defined by
\begin{eqnarray*}
|\pi_0|=\max_{0\leq i\leq m}|s_i-s_{i-1}|.
\end{eqnarray*}
Applying Proposition \ref{L2}, we get
\begin{eqnarray*}
\E\left(|Y_{t}-Y_{s_{i-1}}|^{2}\right)\leq C|\pi_0|,\;\; t\in[s_{i-1},s_i),\;i=1,\cdot\cdot\cdot,m,
\end{eqnarray*}
which together with Burkölder-Davis-Gundy inequality implies
\begin{eqnarray}
\E\left[\max_{1\leq i\leq m}\sup_{s_{i-1}\leq t\leq s_i}|Y_{t}-Y_{s_{i-1}}|^{2}\right]\leq C|\pi_0|.\label{EC}
\end{eqnarray}
The estimate for $Z$ is a little involved. This part will be divide in two steps.

{\bf Step 1}\newline
First we assume that $b,\sigma,\, f,\ g\in C^{1}_{b}$. It follows from Lemma \ref{Lemap} together with Proposition 2.4 that
\begin{eqnarray}
\lim_{|\pi_0|\rightarrow 0}\E\left\{\sup_{0\leq t\leq T}|{Y}^{\pi}_{t}-Y_{t}|^{2}+\int^{T}_{0}|{Z}^{\pi}_{t}-Z_{t}|^{2}dt\right\}=0.\label{b1}
\end{eqnarray}
On the other hand, according to \eqref{apz}, $\tilde{Z}^{\pi_0}_{s_{i-1}}\in L^{2}(\Omega,\mathcal{F}_{s_{i-1}})$. Then since $Z^{\pi}_{s_{i-1}}\in L^{2}(\Omega,\mathcal{F}_{s_{i-1}})$, it follows from Lemma 3.4.2 of \cite{Z}, page 71, that
\begin{eqnarray}
&&\E\left[\sum_{i=1}^{m}\int^{s_i}_{s_{i-1}}|Z_{s}-\tilde{Z}^{\pi_0}_{s_{i-1}}|^{2}ds\right]\leq \E\left[\sum_{i=1}
^{m}\int^{s_i}_{s_{i-1}}|Z_{s}-{Z}^{\pi}_{s_{i-1}}|^{2}ds\right]\nonumber\\
&\leq&
2\E\left[\sum_{i=1}^{m}\int^{s_i}_{s_{i-1}}(|Z_{s}-{Z}^{\pi}_{s}|^{2}+|{Z}^{\pi}_{s}-{Z}^{\pi}_{s_{i-1}}|^{2})ds\right]\nonumber\\
&\leq & C|\pi_0|+\E\left[\sum_{i=1}^{m}\int^{s_i}_{s_{i-1}}|{Z}^{\pi}_{s}-{Z}^{\pi}_{s_{i-1}}|^{2}ds\right].
\label{b'2}
\end{eqnarray}
By \eqref{b1} and \eqref{b'2}, it remains to prove that
\begin{eqnarray}
\sum_{i=1}^{m}\E\left[\int^{s_i}_{s_{i-1}}|{Z}^{\pi}_{s}-{Z}^{\pi}_{s_{i-1}}|^{2}ds\right]\leq C|\pi_0|,\label{R}
\end{eqnarray}
where $C$ is independent of $\pi$ or $\pi_0$. Now we fix $i_0$. For $t\in [s_{i_0-1},s_{i_0})$, it follows from Proposition 2.3 of \cite{PP1} together with Lemma 3.3 that
\begin{eqnarray*}
{Z}^{\pi}_{t}=\left[\left(\xi^{0}_{t}+\sum_{j\geq i}\xi^{j}_{t}\right)M^{-1}_{t}-\int_{0}^{t}f_{x}({\Theta}^{\pi}_{r})\nabla X_{r}
N
^{-1}_{r}dr-\int_{0}^{t}g_{x}({\Xi}^{\pi}_{r})\nabla X_{r}N^{-1}_{r}d\overleftarrow{B}_r\right]N_{t}[\nabla
 X_{t}]^{-1}\sigma(X_{t}).
\end{eqnarray*}
Therefore,
\begin{eqnarray}
|{Z}^{\pi}_{t}-{Z}^{\pi}_{s_{i_0-1}}|\leq I_{t}^{1}+I_{t}^{2}+I_{t}^{3}+I_t^{4}\label{Z}
\end{eqnarray}
where (recalling that $s_{i_0-1}=t_{l_{i_0-1}})$
\begin{eqnarray*}
I_{t}^{1}&=&\left|[\xi_{t}^{0}+\sum_{j\geq i}\xi_{t}^{j}]-[\xi_{s_{i_0-1}}^{0}+\sum_{j\geq l_{i_0-1}+1}\xi_{s_{i_0-1}}^{j}]\right|
\times \left|M_{s_{i_0-1}}^{-1}N_{s_{i_0-1}}[\nabla X_{s_{i_0-1}}]^{-1}\sigma(X_{s_{i_0-1}})\right|,\\
I_{t}^{2}&=&\left|\xi_{t}^{0}+\sum_{j\geq i}\xi_{t}^{j}\right|\left|M_{t}^{-1}N_{t}[\nabla X_{t}]^{-1}\sigma(X_{t})-M_{s_{i_0-1}}^{-1}N_{s_{i_0-1}}[\nabla
 X_{s_{i_0-1}}]^{-1}\sigma(X_{s_{i_0-1}})\right|,\\
 I_{t}^{3}&=&\left|A_t^{1}\right|,\\
I^{4}_t&=&\left|A^{2}_t\right|.
\end{eqnarray*}
with
\begin{eqnarray*}
A_t^{1}&=&\left(\int_{0}^{t}f_{x}({\Theta}^{\pi}_r)\nabla X_{r}N^{-1}_{r}dr\right)N_{t}[\nabla X_{t}]^{-1}\sigma(X_{t})\\
&&-\left(\int_{0}^{s_{i_0-1}}f_{x}({\Theta}^{\pi}_r)\nabla X_{r}N^{-1}_{r}dr\right)N_{s_{i_0-1}}[\nabla X_{s_{i_0-1}}]^{-1}\sigma(X_{s_{i_0-1}})
\end{eqnarray*}
and
\begin{eqnarray*}
A_{t}^{2}&=&\left(\int_{0}^{t}g_{x}({\Xi}^{\pi}_r)\nabla X_{r}N^{-1}_{r}\overleftarrow{dB}_r\right)N_{t}[\nabla X_{t}]^{-1}\sigma(X_{t})\\
&&-\left(\int_{0}^{t_{s_0-1}}g_{x}({\Xi}^{\pi}_r)\nabla X_{r}N^{-1}_{r}\overleftarrow{dB}_r\right)
N_{s_{i_0-1}}[\nabla X_{s_{i_0-1}}]^{-1}\sigma(X_{s_{i_0-1}})
\end{eqnarray*}

Recalling \eqref{m}, and noting that $f_y,\, f_z$ and $g_y$ are uniformly bounded, one can deduce that, for all $p\geq 1$, there exists a constant $C_p$ depending only on $T,\, K$ and $p$, such that
\begin{eqnarray}
\E\left(\sup_{0\leq t\leq T}|N_{t}|^{p}+|N_{t}^{-1}|^{p}\right)&\leq &C_{p};\;\;\;\;
\E\left(\sup_{0\leq t\leq T}[|M_{t}|^{p}
+|M_{t}^{-1}|^{p}]\right)\leq C_{p};\nonumber\\
\nonumber\\
\E\left(|N_{t}-N_s|^{p}+|N_{t}^{-1}-N^{-1}_s|^{p}\right)&\leq& C_{p}|t-s|^{p/2};\label{SE}\\
\nonumber\\
\E\left(|M_{t}-M_s|^{p}+|M_{t}^{-1}-M^{-1}_s|^{p}\right)&\leq& C_{p}|t-s|^{p/2}.
\nonumber
\end{eqnarray}
Thus, applying Proposition 2.1 and 2.2 one can show that
\begin{eqnarray}
 \E(|I_{t}^{3}|^{2})\leq C|\pi_0|,\label{I1}
\end{eqnarray}
and
\begin{eqnarray}
 \E(|I_{t}^{4}|^{2})\leq C|\pi_0|,\label{I01}
\end{eqnarray}
Recalling \eqref{mart} and \eqref{bound} we have
\begin{eqnarray*}
|\xi^{0}_t+\sum_{j\geq i}\xi_t^j|\leq C\E\left\{\sup_{0\leq t\leq T}\nabla X_t|\mathcal{F}_t\right\}.
\end{eqnarray*}
Thus by using again Proposition 2.1 and 2.2 together with \eqref{SE}, we get
\begin{eqnarray}
\E(|I_t^2|^2)\leq C|\pi_0|.\label{I2}
\end{eqnarray}
As proved in \cite{Zhang} (see proof of theorem 3.1), we have
\begin{eqnarray}
\E(|I_t^1|^2)\leq C|\pi_0|.\label{I3}
\end{eqnarray}
Combining \eqref{I1}, \eqref{I01}, \eqref{I2} and \eqref{I3}, we deduce from \eqref{Z} that \eqref{R} holds,
which ends the proof for the smooth case.

{\bf Step 2}\newline  Let consider the general case i.e $b,\, \sigma,\, f,\, g$ are only Lipschitz. For $\varphi=b,\;\sigma,\; f,\;g$, it not difficult to construct via a convolution method, for any $\varepsilon>0$, the function $\varphi^{\varepsilon}\in C^1_b$ be the smooth mollifiers of $\varphi$ such that the derivatives of $\varphi^{\varepsilon}$ are uniformly bounded by $K$ and $\lim_{\varepsilon\rightarrow 0}\varphi^{\varepsilon}=\varphi$. Let $(X^{\varepsilon},Y^{\varepsilon},Z^{\varepsilon})$ and $(X^{\varepsilon},Y^{\pi,\varepsilon},Z^{\pi,\varepsilon})$ denote the solution to corresponding FBDSDE
replaced $\varphi$ by $\varphi^{\epsilon}$ and set
\begin{eqnarray*}
N_{t}^{\varepsilon}&=&\exp\left(\int_{0}^{t}f^{\varepsilon}_y({\Theta}^{\pi,\varepsilon}_r)dr+\int_{0}^{t}g^{\varepsilon}_y({\Xi}^{\pi,\varepsilon}_r) \overleftarrow{dB}_r
-\frac{1}{2}\int_0^t |g^{\varepsilon}_y({\Xi}^{\pi,\varepsilon}_r)|^{2}dr\right),\nonumber\\
M_t^{\varepsilon}&=&\exp\left\{\int_{0}^{t}f^{\varepsilon}_z({\Theta}^{\pi,\varepsilon}_r)dW_r-\frac{1}{2}\int_0^t |f^{\varepsilon}_z({\Theta}^{\pi,\varepsilon}_r)|^{2}dr\right\}.
\end{eqnarray*}
Then one can derive, since the function $f^{\varepsilon}_y, f^{\varepsilon}_z$ and $g^{\varepsilon}_y$ are uniformly bounded by $K$, with the standard calculus about BSDEs, that, for all $p\geq 1$, there exists a constant $C_p$ independent on $\varepsilon$ (depending only on $T,\, K$ and $p$), such that
\begin{eqnarray*}
\E\left(\sup_{0\leq t\leq T}|N^{\varepsilon}_{t}|^{p}+|(N^{\varepsilon})_{t}^{-1}|^{p}\right)&\leq &C_{p};\;\;\;\;
\E\left(\sup_{0\leq t\leq T}[|M^{\varepsilon}_{t}|^{p}
+|(M^{\varepsilon})_{t}^{-1}|^{p}]\right)\leq C_{p};\nonumber\\
\nonumber\\
\E\left(|N^{\varepsilon}_{t}-N^{\varepsilon}_s|^{p}+|(N^{\varepsilon})_{t}^{-1}-(N^{\varepsilon})^{-1}_s|^{p}\right)&\leq& C_{p}|t-s|^{p/2};\label{SE}\\
\nonumber\\
\E\left(|M^{\varepsilon}_{t}-M^{\varepsilon}_s|^{p}+|(M^{\varepsilon})_{t}^{-1}-(M^{\varepsilon})^{-1}_s|^{p}\right)&\leq& C_{p}|t-s|^{p/2}.
\nonumber
\end{eqnarray*}
Next, define
\begin{eqnarray*}
\tilde{Z}^{\varepsilon,\pi_0}_{s_{i-1}}=\frac{1}{s_i-s_{i-1}}\E\left[\int^{s_i}_{s_{i-1}}
Z^{\varepsilon}_sds|\mathcal{F}_{s_{i-1}}\right],
\end{eqnarray*}
we are in the statement of Step 1 from which we deduce that
\begin{eqnarray*}
\sum_{i=1}^{m}\int^{s_i}_{s_{i-1}}|Z_{s}^{\varepsilon}-\tilde{Z}^{\pi_0,\varepsilon}_{t_{i-1}}|^{2}ds\leq C|\pi_0|.\label{G1}
\end{eqnarray*}
Therefore using again Lemma 3.4.2 of \cite{Z}, page 71, we obtain
\begin{eqnarray}
\sum_{i=1}^{m}\E\left[\int^{s_i}_{s_{i-1}}|Z^{\pi}_{s}-\tilde{Z}^{\pi_0}_{s_{i-1}}|^{2}ds\right]&\leq& \sum_{i=1}^{m}\E\left[\int^{s_i}_{s_{i-1}}|Z^{\pi}_{s}-\tilde{Z}^{\pi_0,\varepsilon}_{s_{i-1}}|^{2}ds\right]\nonumber\\
&\leq&\sum_{i=1}^{m}\E\left[\int^{s_i}_{s_{i-1}}[|Z_{s}-Z^{\varepsilon}_{s}|^{2}+|Z_{s}^{\varepsilon}
-\tilde{Z}^{\pi_0,\varepsilon}_{t_{i-1}}|^{2}]ds\right]\nonumber\\
&\leq&\E\left[\int^{T}_{0}|Z_{s}-Z^{\varepsilon}_{s}|^{2}ds\right]+C|\pi_0|.\label{G2}
\end{eqnarray}
Applying Proposition 2.4 we have
\begin{eqnarray*}
\lim_{\varepsilon\rightarrow 0}\E\left[\int^{T}_{0}|Z_{s}-Z^{\varepsilon}_{s}|^{2}ds\right]=0,
\end{eqnarray*}
which, combined with \eqref{G2}, proves the theorem.
\end{proof}

\section{Numerical scheme and rate of convergence}
\setcounter{theorem}{0} \setcounter{equation}{0}
In this section, we consider the BDSDE \eqref{FBSDE} in the special case $\Phi(X)=h(X_T)$
where $h\in W^{1,\infty}(\R^{d})$ such that $h(0)$ is bounded by $K$. The goal of this section is to construct an approximation of the solution $(X,Y,Z)$ by using the "step processes". Let recall
$\pi:\, t_{0}<t_{1}<.....<t_{n}=T$ the partition of $[0,T]$ and $|\pi|=\max_{1\leq i\leq n}|\bigtriangleup^{\pi}_{i}|$, with $\bigtriangleup^{\pi}_{i}=t_{i}-t_{i-1}$. We set also $\bigtriangleup^{\pi}W_{i}=W_{t_{i}}-W_{t_{i-1}},\;\;
\bigtriangleup^{\pi}B_{i}=B_{t_{i}}-B_{t_{i-1}}$
and for all $0\leq i\leq n$ define
$$\mathcal{F}^{\pi}_{i}=\sigma(X_{t_j},\; j\leq i)\vee \mathcal{F}^{B}_{t_i},$$
the discrete-time filtration.
Let briefly review the Euler scheme for the forward diffusion $X$. Define $\pi(t)=t_{i-1}$, for $t\in [t_{i-1},t_i)$. Let $X^{\pi}$ be the solution of the following SDE:
\begin{eqnarray}
X^{\pi}_{t}&=&x+\int_{0}^{t}b(\pi(s),X^{\pi}_{\pi(s)})ds+\int_{0}^{t}\sigma(\pi(s),X^{\pi}_{\pi(s)})dW_{s},\label{SDEnum}
\end{eqnarray}
and we define a "step process" $\hat{X}^{\pi}$ as follows.
\begin{eqnarray}
\hat{X}^{\pi}_t=X^{\pi}_{\pi(t)},\;\;\;\; t\in [0,T].\label{step1}
\end{eqnarray}
The following estimate is well known (see e.g Kloeden and Platen, \cite{KP}).
\begin{proposition}
Assume $b$ and $\sigma$ satisfy the assumptions $({\bf H1})$ and $({\bf H3})$.
Then there exists a constant $C$ depending  only on $T$ and $K$, such that
\begin{eqnarray*}
\max_{1\leq i\leq n} \E\left[\sup_{0\leq t\leq T}|X^{\pi}_{t}-X_{t}|^{2}+\sup_{t_{i-1}\leq t\leq t_{i}}|X_{t}-X_{t_{i-1}}|^{2}\right]\leq C|\pi|. \label{a4}
\end{eqnarray*}
\end{proposition}
Moreover, we get the following estimate involving the step process $\hat{X}^{\pi}$ due to Zhang in \cite{Zhang}.
\begin{proposition}
Assume $b$ and $\sigma$ satisfy the assumptions $({\bf H1})$ and $({\bf H3})$.
Then there exists a constant $C$ depending  only on $T$ and $K$, such that
\begin{eqnarray*}
\sup_{0\leq t\leq T}\E\left[|\hat{X}^{\pi}_{t}-X_{t}|^{2}\right]&\leq& C|\pi|;\\
\E\left[\sup_{0\leq t\leq T}|\hat{X}^{\pi}_{t}-X_{t}|^{2}\right]&\leq& C|\pi|\log\left(\frac{1}{|\pi|}\right).
\end{eqnarray*}
\end{proposition}
The backward component $(Y,Z)$ will be approximated by the following numerical scheme:
\begin{eqnarray}
&&Y_{t_n}^{\pi}= h(X^{\pi}_{T}),\; Z_{t_n}^{\pi}=0\nonumber\\\nonumber\\
&&Z^{\pi}_{t_{i-1}}=\frac{1}{\bigtriangleup_{i}^{\pi}}\E^{\pi}_{i-1}[\tilde{Y}_{t_{i}}^{\pi}\bigtriangleup^{\pi}W_{i}],\label{a5}\\\nonumber\\
&&Y^{\pi}_{t_{i-1}}=\E_{i-1}^{\pi}[\tilde{Y}_{t_{i}}^{\pi}]
+f(t_{i-1},X^{\pi}_{t_{i-1}},Y^{\pi}_{t_{i-1}},Z^{\pi}_{t_{i-1}})\bigtriangleup_{i}^{\pi},\label{a5'}
\end{eqnarray}
where $\displaystyle{\E_{i}^{\pi}[.]=\E[.|\mathcal{F}_{i}^{\pi}]}$ and $\tilde{Y}_{t_{i}}^{\pi}=Y_{t_{i}}^{\pi}+g(t_{i},X^{\pi}_{t_{i}},Y^{\pi}_{t_{i}})\bigtriangleup^{\pi}B_{i}$.
\begin{remark}
\begin{description}
\item $(i)$ Our approximation scheme differ from the one appearing in \cite{Tal}. Indeed, actually we use the conditional expectation with respect
the enlarge filtration $\sigma(X_j,\; j\leq i)\vee \mathcal{F}^{B}_{t_i}$, which is necessary to extend Itô representation theorem
for backward doubly SDE (see Pardoux and Peng, \cite{PP1}).
\item $(ii)$ The backward component and the associated control $(Y, Z)$, which solves the backward doubly SDE, can be expressed as a
function of $X$ and $B$, i.e. $(Y_t, Z_t)=(u(t,B_t, X_t), v(t,B_t, X_t))$, for some deterministic functions $u$ and $v$. Then, the conditional expectations, involved in the above discretization scheme, reduce to the regression of $\tilde{Y}^{\pi}_{t_i}$ and $\tilde{Y}^{\pi}_{t_i} (W_{t_i} - W_{t_{i-1}})$ on the random variable $(X^{\pi}_{t_{i-1}}, B_{t_{i-1}})$ .
\end{description}
\end{remark}

Next, for all $0\leq i\leq n$, on can show that $\tilde{Y}^{\pi}_{t_i}$ belongs to $L^{2}(\Omega,\mathcal{F}_{t_i})$, thus an obvious extension of It\^{o} martingale
representation theorem yields the existence of the $(\mathcal{F}_{s})_{s\in [t_{i-1},t_i)}$-jointly measurable and square integrable process $\bar{Z}^{\pi}$ satisfying
\begin{eqnarray}
\tilde{Y}^{\pi}_{t_{i}}=\E[\tilde{Y}^{\pi}_{t_{i}}|\mathcal{F}^{\pi}_{i-1}]+\int_{t_{i-1}}^{t_{i}}\bar{Z}^{\pi}_{s}dW_{s}.\label{TR}
\end{eqnarray}
Therefore we define the following continuous version
\begin{eqnarray}
Y^{\pi}_{t}&=&Y_{t_{i-1}}^{\pi}-(t-t_{i-1})f(t_{i-1},X^{\pi}_{t_{i-1}},Y^{\pi}_{t_{i-1}},Z^{\pi}_{t_{i-1}})
-g(t_i,X^{\pi}_{t_i},Y^{\pi}_{t_i})(B_{t}-B_{t_{i-1}})\nonumber\\
&&+\int^{t}_{t_{i-1}}\bar{Z}^{\pi}_s dW_s,\;\;\;\;\;\;  t_{i-1}<t\leq t_i .
\label{a6}
\end{eqnarray}
Note that the process $\bar{Z}^{\pi}$ is given by the representation theorem, thus it is useful and even necessary to find a relationship with $Z^{\pi}$
define by \eqref{a5}. We have
\begin{lemma}
\label{L4}
Assume $b,\,\sigma, f$ and $g$ satisfy the assumptions $({\bf H1}),\,({\bf H2})$ and $({\bf H3})$ and let $h\in W^{1,\infty}(\R^d)$ such that $h(0)$ is bounded by $K$.
Then for all $1\leq i\leq n$, we have
\begin{eqnarray*}
Z^{\pi}_{t_{i-1}}=\frac{1}{\bigtriangleup_{i}^{\pi}}\E_{i-1}^{\pi}\left[\int^{t_i}_{t_{i-1}}\bar{Z}^{\pi}_{s}ds\right].
\end{eqnarray*}
\end{lemma}
\begin{proof}
Let recall
\begin{eqnarray*}
\bigtriangleup^{\pi}_{i}Z^{\pi}_{t_{i-1}}&=&\frac{1}{\bigtriangleup_{i}^{\pi}}\E^{\pi}_{i-1}[\left(Y_{t_{i}}^{\pi}
+g(t_i,X_{t_i}^{\pi},Y_{t_{i}}^{\pi})\bigtriangleup^{\pi}
B_i\right)\bigtriangleup^{\pi}W_{i}].
\end{eqnarray*}
Then it follows from $(\ref{TR})$ that
\begin{eqnarray*}
Z^{\pi}_{t_{i-1}}=\frac{1}{\bigtriangleup_{i}^{\pi}}\E^{\pi}_{i-1}\left[\bigtriangleup^{\pi}W_{i}\int^{t_i}_{t_{i-1}}\bar{Z}^{\pi}_s dW_s\right].
\end{eqnarray*}
The result follows by It\^{o}'s isometry.
\end{proof}

We also need the following which is the particular case of Theorem 3.2.
\begin{lemma}
\label{L5}
Assume $b,\,\sigma, f$ and $g$ satisfy the assumptions $({\bf H1}),\,({\bf H2})$ and $({\bf H3})$ and let $h\in W^{1,\infty}(\R^d)$ such that $h(0)$ is bounded by $K$.\newline
Let define, for each $1\leq i\leq n$,
\begin{eqnarray*}
\tilde{Z}^{\pi}_{t_{i-1}}&=&\frac{1}{\bigtriangleup_{i}^{\pi}}\E_{i-1}^{\pi}\left[\int^{t_i}_{t_{i-1}}Z_{s}ds\right].
\label{a'6}
\end{eqnarray*}
Then there exists a constant $C$ depending  only on $T$ and $K$, such that
\begin{eqnarray}
\E\left[\max_{1\leq i\leq n}\sup_{t_{i-1}\leq t\leq t_i}|Y_{t}-Y_{t_{i-1}}|^{2}
+\sum_{i=1}^{n}\int^{t_i}_{t_{i-1}}|Z_{s}
-\tilde{Z}^{\pi}_{t_{i-1}}|^{2}ds\right]\leq C|\pi|.\label{a"6}
\end{eqnarray}
\end{lemma}

We are now ready to state our main result of this section, which provides the rate of convergence of the
numerical scheme \eqref{a5}-\eqref{a5'}.
\begin{theorem}
\label{T1}
Assume $b,\,\sigma, f$ and $g$ satisfy the assumptions $({\bf H1}),\,({\bf H2})$ and $({\bf H3})$ and let $h\in W^{1,\infty}(\R^d)$ such that $h(0)$ is bounded by $K$.
Then there exists a constant $C$ depending  only on $T$ and $K$, such that
\begin{eqnarray*}
\sup_{0\leq t\leq T}\E|Y_{t}-Y_{t}^{\pi}|^{2}+\E\left[\int^{T}_{0}|Z_{s}-\hat{Z}^{\pi}_{s}|^{2}ds\right]\leq C|\pi|.
\end{eqnarray*}
\end{theorem}
\begin{proof}
The proof follows the step of proof of Theorem 3.1 in \cite{Tal} so that we will only outline. In the sequel, $C>0$ will denote the generic constant
independent of $i$ and $n$ and may vary line to line. For $i\in\{0,...,n-1\}$, we set
\begin{eqnarray*}
\delta^{\pi}Y_{t}&=&Y_{t}-Y_{t}^{\pi},\;\; \delta^{\pi}Z_{t}=Z_{t}-\bar{Z}_{t}^{\pi},\;\;\delta^{\pi}f(t)=f(t,X_{t},Y_{t},Z_{t})-f(t_{i},X_{t_{i}}^{\pi},Y_{t_{i}}^{\pi},Z_{t_{i}}^{\pi})\\
&&\mbox{and}\;\; \delta^{\pi}g(t)=g(t,X_{t},Y_{t},)-g(t_{i+1},X_{t_{i+1}}^{\pi},Y_{t_{i+1}}^{\pi}),\;\;\; t\in[t_{i},t_{i+1}).
\end{eqnarray*}
By It\^{o}'s formula, it follows from Lipschitz condition on $f,\ g$ and $h$, together with the
inequality $ab\leq \beta a^{2}+b^{2}/\beta$ that
\begin{eqnarray}
V_t&=&\E|\delta^{\pi}Y_{t}|^{2}+\E\int^{t_{i+1}}_{t}|\delta^{\pi}Z_{s}|^{2}ds-|\delta^{\pi}Y_{t_{i+1}}|^{2}\nonumber\\
&=&2\E\int^{t_{i+1}}_{t}\langle
\delta^{\pi}Y_{s},\delta^{\pi}f(s)\rangle ds +\int^{t_{i+1}}_{t}|\delta^{\pi}g(s)|^{2}ds\nonumber\\
&\leq& \frac{C}{\beta}\int_{t}^{t_{i+1}}\E\left\{|\pi|^{2}+|X_{s}-X^{\pi}_{t_{i}}|^{2}
+|Y_{s}-Y^{\pi}_{t_{i}}|^{2}+|Z_{s}-Z^{\pi}_{t_{i}}|^{2}\right\}ds\nonumber\\
&&+\int_{t}^{t_{i+1}}C\E\left\{|\pi|^{2}+|X_{s}-X^{\pi}_{t_{i+1}}|^{2}
+|Y_{s}-Y^{\pi}_{t_{i+1}}|^{2}\right\}ds\nonumber\\
&&+\beta\int_{t}^{t_{i+1}}\E|\delta^{\pi}Y_{s}|^{2}ds, \;\; t\in[t_{i},t_{i+1}).
\label{T3.2.1}
\end{eqnarray}

Proposition 3.1, Lemma 3.3 and Lemma 3.4 yield that
\begin{eqnarray}
\begin{array}{l}
\E|X_{s}-X^{\pi}_{t_{i}}|^{2}+\E|X_{s}-X^{\pi}_{t_{i+1}}|^{2}\leq C|\pi|,\\\\
\E|Y_{s}-Y^{\pi}_{t_{i}}|^{2}\leq 2\left(\E|Y_{s}-Y_{t_{i}}|^{2}
+\E|\delta^{\pi}Y_{t_{i}}|^{2}\right)\leq C\left(|\pi|+\E|\delta^{\pi}Y_{t_{i}}|^{2}\right)\\\\
\E|Y_{s}-Y^{\pi}_{t_{i+1}}|^{2}\leq 2\left(\E|Y_{s}-Y_{t_{i+1}}|^{2}+\E|\delta^{\pi}Y_{t_{i+1}}|^{2}\right)\leq C\left(|\pi|+\E|\delta^{\pi}Y_{t_{i+1}}|^{2}\right)\\\\
\E|Z_{s}-Z^{\pi}_{t_{i}}|^{2}\leq 2\left(\E|Z_{s}-\tilde{Z}^{\pi}_{t_{i}}|^{2}+\frac{1}{\Delta_{i+1}^{\pi}}\int_{t_{i}}^{t_{i+1}}
\E|\delta^{\pi}Z_{r}|^{2}dr\right)
\end{array}
\label{T3.2.3}
\end{eqnarray}

Plugging $(\ref{T3.2.3})$ into $(\ref{T3.2.1})$, we get
\begin{eqnarray*}
V_{t}&\leq& \frac{C}{\beta}\int_{t}^{t_{i+1}}\E\left\{|\pi|
+|\delta^{\pi}Y_{t_i}|^{2}+|Z_{s}-\tilde{Z}^{\pi}_{t_{i}}|^{2}\right\}ds\nonumber\\
&&+C\int_{t}^{t_{i+1}}\E\left\{|\pi|
+|\delta^{\pi}Y_{t_{i+1}}|^{2}\right\}ds\nonumber\\
&&+\frac{C}{\beta}\int_{t}^{t_{i+1}}\E|\delta^{\pi}Z_{s}|^{2}ds
+\beta\int_{t}^{t_{i+1}}\E|\delta^{\pi}Y_{s}|^{2}ds,
\label{T3.2.5}
\end{eqnarray*}
from which and the definition of $V_{t}$ provide, for $t_{i}\leq t\leq t_{i+1},$
\begin{eqnarray}
\E|\delta^{\pi}Y_{t}|^{2}+\int_{t}^{t_{i+1}}\E|\delta^{\pi}Z_{s}|^{2}ds\leq\beta\int_{t}^{t_{i+1}}\E|\delta^{\pi}Y_{s}|^{2}ds+A_{i}
\label{T3.2.6}
\end{eqnarray}
where
\begin{eqnarray*}
A_{i}&=&(1+C\pi)\E|\delta^{\pi}Y_{t_{i+1}}|^{2}+\frac{C}{\beta}\left[|\pi|^{2}+|\pi|\E|Y^{\pi}_{t_i}|+\int_{t_{i}}^{t_{i+1}}\E|Z_{s}-\tilde{Z}^{\pi}_{t_{i}}|^{2}ds\right]\\
&&+\frac{C}{\beta}\int_{t_{i}}^{t_{i+1}}\E|\delta^{\pi}Z_{s}|^{2}ds.
\end{eqnarray*}
Next, by little calculus used Gronwall's Lemma, we have
\begin{eqnarray}
\E|\delta^{\pi}Y_{t}|^{2}+\int_{t}^{t_{i+1}}\E|\delta^{\pi}Z_{s}|^{2}ds \leq \left(1+C\beta|\pi|\right)A_{i};
\label{gronw}
\end{eqnarray}
 hence for $t=t_i$ and $\beta$ sufficiently large than $C$, such that $\frac{C}{\beta}<1$, we obtain
\begin{eqnarray*}
&&\E|\delta^{\pi}Y_{t_i}|^{2}+(1-\frac{C}{\beta})\int_{t_i}^{t_{i+1}}\E|\delta^{\pi}Z_{s}|^{2}ds\nonumber\\
&\leq& (1+C|\pi|)\left\{\E|\delta^{\pi}Y_{t_{i+1}}|^{2}+|\pi|^{2}+
\int_{t_{i}}^{t_{i+1}}\E[|Z_{s}-\tilde{Z}^{\pi}_{t_{i}}|^{2}]ds\right\}
\end{eqnarray*}
for small $|\pi|$.

Iterating the last inequality, we get
\begin{eqnarray*}
&&\E|\delta^{\pi}Y_{t_i}|^{2}+(1-\frac{C}{\beta})\int_{t_i}^{t_{i+1}}\E|\delta^{\pi}Z_{s}|^{2}ds\\
&\leq& (1+C|\pi|)^{T/|\pi|}\left\{\E|\delta^{\pi}Y_{T}|^{2}+|\pi|
+\sum_{i=1}^{n}\int_{t_{i-1}}^{t_{i}}\E[|Z_{s}-\tilde{Z}^{\pi}_{t_{i-1}}|^{2}]ds\right\}.
\end{eqnarray*}
Moreover, it follows from Lemma 3.4, Lipschitz condition on $g$ and Proposition 3.1 that
\begin{eqnarray}
&&\E|\delta^{\pi}Y_{t_i}|^{2}+(1-\frac{C}{\beta})\int_{t_i}^{t_{i+1}}\E|\delta^{\pi}Z_{s}|^{2}ds\nonumber\\
&\leq& (1+C|\pi|)^{T/\pi}\left\{\E|\delta^{\pi}Y_{T}|^{2}+|\pi|+C|\pi|\right\}\leq C|\pi| \label{est1}
\end{eqnarray}
for small $|\pi|$.

On the other hand summing up inequality \eqref{gronw} with $t=t_i$, we get
\begin{eqnarray*}
&&\left[1-\frac{C}{\beta}(1+C\beta|\pi|)\right]\int_0^T\E|\delta^{\pi} Z_s|^2 ds\\
&\leq& (1+C\beta|\pi|)\frac{C}{\beta}|\pi|+(1+C\beta|\pi|)(1+C|\pi|)\E|\delta^{\pi}Y_T|^2\\
&&+\left[(1+C\beta|\pi|)\frac{C}{\beta}|\pi|-1\right]\E|\delta^{\pi}Y_0|^2\\
&&+\left[(1+C\beta|\pi|)((1+C|\pi|)+\frac{C}{\beta}|\pi|)-1\right]\sum_{i=1}^{n-1}\E|\delta^{\pi}Y_{t_{i}}|^{2}\\
&&+(1+C\beta|\pi|)\frac{C}{\beta}\sum_{i=0}^{n-1}\int_{t_i}^{t_{i+1}}\E|Z_s-\tilde{Z}_{t_i}^{\pi}|^{2}ds.
\end{eqnarray*}
Therefore, by inequality \eqref{est1} and Lemma 3.4 one derives that
\begin{eqnarray*}
\int_0^T\E|\delta^{\pi} Z_s|^2 ds\leq C|\pi|
\end{eqnarray*}
and then
\begin{eqnarray*}
\sup_{0\leq t\leq T}|\delta^{\pi}Y_{t}|^{2}\leq C|\pi|.
\end{eqnarray*}
\end{proof}
To end this section, let give the following bound on $Y^{\pi}_{y_i}$'s which will be used in the approximating of discrete conditional expectation $\E^{\pi}_i$, for all $0\leq i\leq n-1$.
\begin{lemma}
Assume $b,\,\sigma, f$ and $g$ satisfy the assumptions $({\bf H1}),\,({\bf H2})$ and $({\bf H3})$ and let $h\in W^{1,\infty}(\R^d)$ such that $h(0)$ is bounded by $K$. For all $0\leq i\leq n-1$, define the sequences of random variables by backward induction
\begin{eqnarray*}
\alpha^{\pi}_n&=&2C,\;\;\; \beta_n=C,\\
\alpha_i^{\pi}&=&(1-C|\pi|)^{-1}(1+C^2|\pi|)^{1/2}\left\{(1+2C|\pi|)[(1+C|\bigtriangleup^{\pi}B_{i+1}|)
\beta_{i+1}^{\pi}+C|\bigtriangleup^{\pi}B_{i+1}|]+C|\pi|\right\} \\
\beta^{\pi}_i&=&(1-C|\pi|)^{-1}(1+C^2|\pi|)^{1/2}\left\{(1+C|\bigtriangleup^{\pi}B_{i+1}|)\alpha^{\pi}_{i+1}+6C^2|\pi|
(1+2C|\bigtriangleup^{\pi}B_{i+1}|)+3C|\pi|\right\}                ,
\end{eqnarray*}
Then, for all $0\leq i\leq n$;
\begin{eqnarray}
|Y^{\pi}_{t_{i}}|&\leq& \alpha_i^{\pi}+\beta_i^{\pi}|X^{\pi}_{t_{i}}|^2,\label{bound1}\
\end{eqnarray}
\begin{eqnarray}
&&\E^{\pi}_{i-1}|Y^{\pi}_{t_{i}}+g(t_{i},X^{\pi}_{t_{i}},Y^{\pi}_{t_{i}})\bigtriangleup^{\pi}B_{i}|\nonumber\\
&\leq &(\E^{\pi}_{i-1}|Y^{\pi}_{t_{i}}
+g(t_{i},X^{\pi}_{t_{i}},Y^{\pi}_{t_{i}})\bigtriangleup^{\pi}B_{i}|^{2})^{1/2}\nonumber\\
&\leq& (1+2C|\pi|)\left\{(1+C|\bigtriangleup^{\pi}B_{i+1}|)\beta_{i+1}^{\pi}+C|\bigtriangleup^{\pi}B_{i+1}|\right\}
|X^{\pi}_{t_{i}}|^{2}\nonumber\\
&&+(1+C|\bigtriangleup^{\pi}B_{i+1}|)\alpha^{\pi}_{i+1}+6C^2|\pi|(1+2C|\bigtriangleup^{\pi}B_{i+1}|)label{bound2}
\end{eqnarray}
\begin{eqnarray}
&&\E^{\pi}_{i-1}|(Y^{\pi}_{t_{i}}+g(t_{i},X^{\pi}_{t_{i}},Y^{\pi}_{t_{i}})\bigtriangleup^{\pi}B_{i})
\bigtriangleup^{\pi}W_{i+1}|\nonumber\\
&\leq&\sqrt{|\pi|}(1+2C|\pi|)\left\{(1+C|\bigtriangleup^{\pi}B_{i+1}|)\beta_{i+1}^{\pi}+C|\bigtriangleup^{\pi}
B_{i+1}|\right\}|X^{\pi}_{t_{i}}|^{2}\nonumber\\
&&+\sqrt{|\pi|}(1+C|\bigtriangleup^{\pi}B_{i+1}|)\alpha^{\pi}_{i+1}+6C^2|\pi|(1+2C|\bigtriangleup^{\pi}B_{i+1}|).
\label{bounded 3}
\end{eqnarray}
Moreover,
\begin{eqnarray*}
\limsup_{|\pi|\rightarrow 0}\max_{0\leq i\leq n}(\alpha^{\pi}_i+\beta^{\pi}_i)<\infty, \; a.s.
\end{eqnarray*}
\end{lemma}
\begin{proof}
First, since $h$ is $C$-Lipschitz, $h(0)$ is bounded by $C$,
\begin{eqnarray*}
|Y_{t_n}|=|h(X^{\pi}_T)|\leq C(|X^{\pi}_T|+1)\leq 2C+C|X^{\pi}_T|^{2}=\alpha_n+\beta_n|X^{\pi}_T|^2.
\end{eqnarray*}
Next, we assume that
\begin{eqnarray}
|Y_{t_{i+1}}|\leq \alpha_{i+1}^{\pi}+\beta_{i+1}^{\pi}|X^{\pi}_{t_{i+1}}|^2,
\end{eqnarray}
for some fixed $0\leq i\leq n-1$. Then by the definition of $Y^{\pi}$ in \eqref{a5'}, there exists a $\mathcal{F}_{t_i}$-measurable random
variable $\zeta_i$ such that
\begin{eqnarray}
(1-C|\pi|)|Y^{\pi}_{t_i}|&\leq &\E_{i}^{\pi}[(Y^{\pi}_{t_{i+1}}+g(t_{i+1},X^{\pi}_{t_{i+1}},Y^{\pi}_{t_{i+1}})\bigtriangleup^{\pi}B_{i+1})(1
+\zeta_i\bigtriangleup^{\pi}W_{i+1}]+C|\pi|(2+|X^{\pi}_{t_{i}}|)\nonumber\\\nonumber\\
&\leq &(\E_{i}^{\pi}|Y^{\pi}_{t_{i+1}}+g(t_{i+1},X^{\pi}_{t_{i+1}},Y^{\pi}_{t_{i+1}})\bigtriangleup^{\pi}B_{i+1}|^{2})^{1/2}(\E_i^{\pi}|
+\zeta_{i}\bigtriangleup^{\pi}W_{i+1}|^{2})^{1/2}\nonumber\\\nonumber\\
&&+C|\pi|(3+|X^{\pi}_{t_{i}}|^{2}).\label{bound4}
\end{eqnarray}
It show in \cite{Tal} that
\begin{eqnarray*}
\E_i^{\pi}|1+\zeta_{i}\bigtriangleup^{\pi}W_{i+1}|^{2}\leq 1+C^2|\pi|.
\end{eqnarray*}
This provide from \eqref{bound4}
\begin{eqnarray}
(1-C|\pi|)|Y^{\pi}_{t_i}|&\leq & (1+C^2|\pi|)^{1/2}(\E_{i}^{\pi}|Y^{\pi}_{t_{i+1}}+g(t_{i+1},X^{\pi}_{t_{i+1}},
Y^{\pi}_{t_{i+1}})\bigtriangleup^{\pi}B_{i+1}|^{2})^{1/2}\nonumber\\\nonumber\\
&&+C|\pi|(3+|X^{\pi}_{t_{i}}|^{2}).\label{bound5}
\end{eqnarray}
But it follows from the Lipschitz property of $g$ that,
\begin{eqnarray*}
&&\E_i^{\pi}(|Y^{\pi}_{t_{i+1}}+g(t_{i+1},X^{\pi}_{t_{i+1}},Y^{\pi}_{t_{i+1}})\bigtriangleup^{\pi}B_{i+1}|^{2})^{1/2}\\
&\leq& (1+C|\bigtriangleup^{\pi}B_{i+1}|)(\E_i^{\pi}|Y^{\pi}_{t_{i+1}}|^2)^{1/2}
+C|\bigtriangleup^{\pi}B_{i+1}|(\E_i^{\pi}|X^{\pi}_{t_{i+1}}|^2)^{1/2}\\
&\leq&(1+C|\bigtriangleup^{\pi}B_{i+1}|)\left\{\alpha^{\pi}_{i+1}+\beta_{i+1}^{\pi}[(1+2C|\pi|)|X^{\pi}_{t_{i}}|^{2}+6K^2|\pi|]\right\}\\
&&+C|\bigtriangleup^{\pi}B_{i+1}|[(1+2C|\pi|)|X^{\pi}_{t_{i}}|^{2}+6K^2|\pi|]\\
&=&(1+2C|\pi|)\left\{(1+C|\bigtriangleup^{\pi}B_{i+1}|)\beta_{i+1}^{\pi}+C|\bigtriangleup^{\pi}B_{i+1}|\right\}|X^{\pi}_{t_{i}}|^{2}\\
&&+(1+C|\bigtriangleup^{\pi}B_{i+1}|)\alpha^{\pi}_{i+1}+6C^2|\pi|(1+2C|\bigtriangleup^{\pi}B_{i+1}|).
\end{eqnarray*}
Finally \eqref{bound5} becomes
\begin{eqnarray*}
|Y^{\pi}_{t_i}|&\leq &(1-C|\pi|)^{-1}(1+C^2|\pi|)^{1/2}\\
&&\times\left\{(1+2C|\pi|)[(1+C|\bigtriangleup^{\pi}B_{i+1}|)
\beta_{i+1}^{\pi}+C|\bigtriangleup^{\pi}B_{i+1}|]+C|\pi|\right\}|X^{\pi}_{t_{i}}|^{2}\nonumber\\
&&+(1-C|\pi|)^{-1}(1+C^2|\pi|)^{1/2}\\
&&\times\left\{(1+C|\bigtriangleup^{\pi}B_{i+1}|)\alpha^{\pi}_{i+1}+6C^2|\pi|(1+2C|\bigtriangleup^{\pi}B_{i+1}|)+3C|\pi|\right\}\\
&=&\alpha^{\pi}_i+\beta^{\pi}_i|X^{\pi}_{t_{i}}|^{2}.
\end{eqnarray*}
\end{proof}
\section {Rate of convergence of the regression approximation}
In this section, we try to give some ideas for method of simulating numerical scheme derived in the above section. It well know that the process $X^{\pi}$ defined
by \eqref{SDEnum} is simulated by the classical Monte-Carlo method. We are reduced to simulate the process $(Y^{\pi},Z^{\pi})$ defined in \eqref{a5} and \eqref{a5'}.
In practice, the main tool to define of an approximation of $Y^{\pi}$, and then of $Z^{\pi}$, is to replace the conditional expectation $\E^{\pi}_i$
by its estimator $\widehat{E}^{\pi}_i$ in the backward scheme \eqref{a5} and \eqref{a5'}. We first establish the following bound on the $Y^{\pi}_{t_i}$'s which
help us to derive this simulation.

For the regression approximation, we consider
$\{\mathcal{P}^{\pi}_i\}_{0\leq i\leq n}, \,\{\mathcal{R}^{\pi}_{\ i}\}_{0\leq i\leq n}$ and $\{\mathcal{J}^{\pi}_i\}_{0\leq i\leq n}$ defined by:
\begin{eqnarray*}
{\mathcal{P}}^{\pi}_i&=&\alpha_i^{\pi}+\beta_i^{\pi}|X^{\pi}_{t_{i}}|^2\\
\mathcal{R}^{\pi}_i&=&(1+2C|\pi|)\left\{(1+C|\bigtriangleup^{\pi}B_{i+1}|)\beta_{i+1}^{\pi}+C|\bigtriangleup^{\pi}B_{i+1}|\right\}|X^{\pi}_{t_{i}}|^{2}\\
&&+(1+C|\bigtriangleup^{\pi}B_{i+1}|)\alpha^{\pi}_{i+1}+6C^2|\pi|(1+2C|\bigtriangleup^{\pi}B_{i+1}|)\\
\mathcal{J}^{\pi}_i&=& \sqrt{|\pi|}(1+2C|\pi|)\left\{(1+C|\bigtriangleup^{\pi}B_{i+1}|)\beta_{i+1}^{\pi}+C|\bigtriangleup^{\pi}B_{i+1}|\right\}|X^{\pi}_{t_{i}}|^{2}\nonumber\\
&&+\sqrt{|\pi|}(1+C|\bigtriangleup^{\pi}B_{i+1}|)\alpha^{\pi}_{i+1}+6C^2|\pi|(1+2C|\bigtriangleup^{\pi}B_{i+1}|).
\end{eqnarray*}
Therefore thanks to Lemma 4.7, we have
\begin{eqnarray}
-\mathcal{P}^{\pi}_i(X^{\pi}_{t_{i}},\bigtriangleup^{\pi}B_{i+1})\leq Y^{\pi}_{t_i}\leq\mathcal{P}^{\pi}_i(X^{\pi}_{t_{i}}
,\bigtriangleup^{\pi}B_{i+1})\label{regr1}\\\nonumber\\
-\mathcal{R}^{\pi}_i(X^{\pi}_{t_{i}},\bigtriangleup^{\pi}B_{i+1})\leq \E_{i}^{\pi}[\tilde{Y}^{\pi}_{t_{i+1}}]\leq\mathcal{R}^{\pi}_{\ i}(X^{\pi}_{t_{i}}
,\bigtriangleup^{\pi}B_{i+1})\label{regr2}\\\nonumber\\
-\mathcal{J}^{\pi}_i(X^{\pi}_{t_{i}},\bigtriangleup^{\pi}B_{i+1})\leq \E[\tilde{Y}^{\pi}_{t_{i+1}}\bigtriangleup^{\pi}W_{i+1}]\leq\mathcal{J}^{\pi}_i(X^{\pi}_{t_{i}}
,\bigtriangleup^{\pi}B_{i+1}).\label{regr3}
\end{eqnarray}
Next, for a $\R$-valued random variable $\xi$, we define
\begin{eqnarray*}
{\bf T}_i^{\mathcal{P}^{\pi}}(\xi)&=&-\mathcal{P}^{\pi}_i(X^{\pi}_{t_i},\bigtriangleup^{\pi}B_{i+1})\vee
\xi\wedge\mathcal{P}^{\pi}_i(X^{\pi}_{t_i},\bigtriangleup^{\pi}B_{i+1})\\\\
{\bf T}_i^{\mathcal{R}^{\pi}}(\xi)&=&-\mathcal{R}^{\pi}_i(X^{\pi}_{t_i},\bigtriangleup^{\pi}B_{i+1})\vee
\xi\wedge\mathcal{R}^{\pi}_{\ i}(X^{\pi}_{t_i},\bigtriangleup^{\pi}B_{i+1})\\\\
{\bf T}_i^{\mathcal{J}^{\pi}}(\xi)&=&-\mathcal{J}^{\pi}_i(X^{\pi}_{t_i},\bigtriangleup^{\pi}B_{i+1})\vee
\xi\wedge\mathcal{J}^{\pi}_i(X^{\pi}_{t_i},\bigtriangleup^{\pi}B_{i+1}).
\end{eqnarray*}
Given an approximation $\widehat{\E}^{\pi}_i$ of $\E^{\pi}_i$, we are ready to get the process $(\hat{Y}^{\pi}, \hat{Z}^{\pi})$ defined by following backward induction
scheme:
\begin{eqnarray}
\hat{Y}^{\pi}_{t_n}&=&g(X^{\pi}_T),\label{appr1}\\\nonumber\\
\hat{Z}^{\pi}_{t_{i-1}}&=&\frac{1}{\bigtriangleup_{i}^{\pi}}\widehat{\E}^{\pi}_{i-1}[\left(\hat{Y}_{t_{i}}^{\pi}
+g(t_i,X_{t_i}^{\pi},\hat{Y}_{t_{i}}^{\pi})\bigtriangleup^{\pi}B_i\right)\bigtriangleup^{\pi}W_{i}]\label{appr2}\\\nonumber\\
{Y}^{\pi}_{t_{i-1}}&=&\widehat{\E}_{i-1}^{\pi}[\hat{Y}_{t_{i}}^{\pi}+g(t_{i},X^{\pi}_{t_{i}},\hat{Y}^{\pi}_{t_{i}})\bigtriangleup^{\pi}B_{i}]
+f(t_{i-1},X^{\pi}_{t_{i-1}},{Y}^{\pi}_{t_{i-1}},\hat{Z}^{\pi}_{t_{i-1}})\bigtriangleup_{i}^{\pi}\label{appr3}\\\nonumber\\
\hat{Y}^{\pi}_{t_{i-1}}&=&{\bf T}_{t_{i-1}}^{\mathcal{P}^{\pi}}({Y}^{\pi}_{t_{i-1}})\label{appr4},
\end{eqnarray}
for all $1\leq i\leq n$.
\begin{remark}
Using the above notation and replace $X^{\pi}$ by $U^{\pi}=(X^{\pi},B^{\pi})$, on can state analogous to examples 4.1 and 4.2 appearing in \cite{Tal}.
\end{remark}
To end this section, we derive the following $L^p$ estimate of the error $\hat{Y}^{\pi}-Y^{\pi}$ in terms of the regression errors $\widehat{E}^{\pi}_i-\E^{\pi}_i$.
\begin{theorem}
Let $p\geq 1$ be given, and $\mathcal{P}^{\pi}$ be a sequence defined above.
Then, there is a constant $C$ depending only on $T, K$ and $p$ such that
\begin{eqnarray*}
\|\hat{Y}^{\pi}_{t_i}-Y^{\pi}_{t_i}\|_{L^{p}}&\leq &\frac{C}{\pi}\max_{1\leq j\leq n-1}\left\{\|(\widehat{\E}_j-\E_j)[\hat{Y}^{\pi}_{t_{j+1}}
+g(t_{j+1},X_{t_{j+1}}^{\pi},\hat{Y}_{t_{j+1}}^{\pi})\bigtriangleup^{\pi}B_{j+1}]\|_{L^{p}}\right.\\
&&\left.+\|(\widehat{\E}_j-\E_j)[(\hat{Y}^{\pi}_{t_{j+1}}+g(t_{j+1},X_{t_{j+1}}^{\pi},\hat{Y}_{t_{j+1}}^{\pi})\bigtriangleup^{\pi}B_{j+1}) \bigtriangleup^{\pi}W_{j+1}]\|_{L^{p}}\right\}
\end{eqnarray*}
\end{theorem}
\begin{proof}
For $0\leq i\leq n-1$ be fixed, with a similarly calculus as one use in proof of Theorem 4.1 \cite{Tal}, we have
\begin{eqnarray}
(1-C\pi)|Y^{\pi}_{t_i}-\hat{Y}^{\pi}_{t_i}|\leq |\varepsilon_i|+(1+C|\bigtriangleup^{\pi}B_{i+1}|)(\E|Y^{\pi}_{t_i}-\hat{Y}^{\pi}_{t_i}|^{p})^{1/p}(\E_i^{\pi}
|1+\zeta_i\bigtriangleup^{\pi}W_{i+1}|^{2k})^{1/2k},\nonumber\\\label{appr5}
\end{eqnarray}
where $k$ is an arbitrary integer greater than the conjugate of $p$ and
\begin{eqnarray*}
\varepsilon_i&=&(\widehat{\E}_i-\E_i)[\hat{Y}^{\pi}_{t_{i+1}}+g(t_{i+1},X_{t_{i+1}}^{\pi},\hat{Y}_{t_{i+1}}^{\pi})\bigtriangleup^{\pi}B_{i+1}]\\
&& +\bigtriangleup^{\pi}_{i+1}\left\{f(t_{i},X^{\pi}_{t_i},{Y}^{\pi}_{t_i},(\bigtriangleup^{\pi}_{i+1})^{-1}\E^{\pi}_i[(\hat{Y}^{\pi}_{t_{i+1}}
+g(t_{i+1},X_{t_{i+1}}^{\pi},\hat{Y}_{t_{i+1}}^{\pi})\bigtriangleup^{\pi}B_{i+1})\bigtriangleup^{\pi}W_{i+1}])\right.\\
&&\left.-f(t_{i},X^{\pi}_{t_i},{Y}^{\pi}_{t_i},
(\bigtriangleup^{\pi}_{i+1})^{-1}\widehat{\E}^{\pi}_i[(\hat{Y}^{\pi}_{t_{i+1}}
+g(t_{i+1},X_{t_{i+1}}^{\pi},\hat{Y}_{t_{i+1}}^{\pi})\bigtriangleup^{\pi}B_{i+1})\bigtriangleup^{\pi}W_{i+1}])\right\}.
\end{eqnarray*}
Since
\begin{eqnarray*}
\|\varepsilon_i\|_{L^p}\leq \eta_i=C(\|(\widehat{\E}_i-\E_i)[\hat{Y}^{\pi}_{t_{i+1}}+g(t_{i+1},X_{t_{i+1}}^{\pi},\hat{Y}_{t_{i+1}}^{\pi})\bigtriangleup^{\pi}B_{i+1}]\|_{L^p}\\
+\|(\widehat{\E}_i-\E_i)[(\hat{Y}^{\pi}_{t_{i+1}}+g(t_{i+1},X_{t_{i+1}}^{\pi},\hat{Y}_{t_{i+1}}^{\pi})\bigtriangleup^{\pi}B_{i+1})\bigtriangleup^{\pi}W_{i+1}]\|_{L^p},
\end{eqnarray*}
and
\begin{eqnarray*}
(\E_i^{\pi}|1+\zeta_i\bigtriangleup^{\pi}W_{i+1}|^{2k})^{1/2k}\leq (1+C|\pi|),
\end{eqnarray*}
we get from \eqref{appr5}
\begin{eqnarray}
(1-C\pi)\|Y^{\pi}_{t_i}-\hat{Y}^{\pi}_{t_i}\|_{L^p}\leq \eta_i+(1+C|\pi|)^{1/2k}\|Y^{\pi}_{t_{i+}}-\hat{Y}^{\pi}_{t_{i+1}}\|_{L^p}
\end{eqnarray}
and thus the result follows as in \cite{Tal}.
\end{proof}
\begin{remark}
With the foregoing results, it is possible with some not very difficult adjustment to obtain similar results as those obtained by Bouchard and Touzi (see Section 5 and 6 of \cite{Tal}).

\end{remark}

\subsection*{Acknowledgments} The author would like to thank I. Boufoussi and Y. Ouknine for their valuable comments and suggestions and express
his deep gratitude to  and UCAM Mathematics Department for their friendly hospitality during his stay in Cadi Ayyad University.\newline
We also thank a anonymous referee whose suggestions and comments have been very useful in improving the original manuscript.

\label{lastpage-01}
\end{document}